\documentclass{amsart}
\usepackage{amssymb}
\usepackage{eucal}
\newcommand{\R}{{\Bbb R}}

\newcommand{\N}{{\Bbb N}}

\usepackage{physics}
\usepackage{amsmath}
\usepackage{tikz}
\usepackage{mathdots}
\usepackage{yhmath}
\usepackage{cancel}
\usepackage{color}
\usepackage{siunitx}
\usepackage{array}
\usepackage{multirow}

\usepackage{gensymb}
\usepackage{tabularx}
\usepackage{extarrows}
\usepackage{booktabs}
\usetikzlibrary{fadings}
\usetikzlibrary{patterns}
\usetikzlibrary{shadows.blur}
\usetikzlibrary{shapes}

\usepackage{amsmath}

\newtheorem{theorem}{Theorem}[section]
\newtheorem{lemma}[theorem]{Lemma}
\newtheorem{cor}[theorem]{Corollary}
\theoremstyle{definition}

\theoremstyle{remark}
\newtheorem{remark}[theorem]{Remark}

\numberwithin{equation}{section}

\begin{document}

\title[Wavefront stability with asymptotic phase]{Wavefront's stability with asymptotic phase in the delayed monostable equations}

\author{Abraham Solar}
\address{DMFA, Universidad Católica de la Santísima Concepción, Concepción, Chile}
\email{asolar@ucsc.cl}

\thanks{This work was supported by FONDECYT  (Chile),   projects  11190350 (A.S.), 1190712 (S.T.).}

\author{Sergei Trofimchuk}
\address{Instituto de Matem\'atica, Universidad de Talca, Casilla 747, Talca, Chile}
\email{trofimch@inst-mat.utalca.cl}

%    General info

\date{July 24, 2021 and, in revised form... }



\begin{abstract} We extend the class of initial conditions for scalar delayed reaction-diffusion equations $u_t (t,x)=u_{xx}(t,x)+f(u(t, x), u(t-h, x))$  which evolve in solutions converging to monostable traveling waves. Our approach allows to compute, in the moving reference frame, the phase distortion $\alpha$ of the limiting travelling wave with respect to the position of solution at the initial moment  $t=0$. In general,  $\alpha\not=0$ for the Mackey-Glass type diffusive equation.  Nevertheless, $\alpha=0$ for the KPP-Fisher delayed equation: the related theorem also improves existing stability conditions for this model.

\end{abstract}

\maketitle

\section{Introduction: main results and applications}

The previous studies (e.g. see \cite{BS,Chern,LvW,wlr}) show that both minimal and non-minimal positive traveling waves\footnote{By definition, the profile  $\phi$ should satisfy  $ \phi(-\infty)=0, 
$ $\liminf_{t \to +\infty}\phi(t)>0$,  $\sup_{t \in \R} \phi(t) <\infty$.} $u(t,x) = \phi(x+ct)$ for the monostable delayed reaction-diffusion equation 
\begin{equation}\label{E} 
u_t (t,x)=u_{xx}(t,x)+f(u(t, x),u(t-h, x)), \quad t>0, \  x\in{\mathbb R},
\end{equation}
attract solutions\footnote{We assume everywhere that  (i) $u_0(s,x)$ is bounded, globally Lipschitz  continuous in $x$  (uniformly in $s$) and (ii)  the solution 
$u(t, x)$  exists globally and is bounded on the strips $[0,n]\times \R, n \in \N$. Note that (ii) is satisfied automatically for both models (KPP-Fisher and Nicholson's) of the paper.} $u(t,x)$ whose initial segments $u_0(s, x)$  have the same leading  asymptotic terms  
at $x = -\infty$ as the shifted wave  $\phi(x+cs)$, for all $s\in[-h, 0]$. The latter assumption implies that, for some positive $A_0$, 
\begin{equation}\label{EC}
\lim_{x \to -\infty} \frac{u_0(s,x)}{\phi(x+cs)}=A_0, \quad s \in [-h,0]. 
\end{equation}

This observation concerns so-called pulled waves for equation (\ref{E}) and 
smooth traveling waves for delayed degenerate reaction-diffusion equations \cite{HMD}. The pushed and bistable waves have better stability properties  \cite{SOTRb,SOTRa} and they are not considered in this work. 

Condition (\ref{EC}) seems to be excessively restrictive: for example, 
it excludes  initial segments asymptotically similar, in the spirit of (\ref{EC}),  to $\phi(x+\alpha(s)),$ $ s\in[-h, 0]$, with  nonlinear shift $\alpha (s)$.  This circumstance  is irrelevant for the non-delayed equations when $h=0$, however, in the delayed case it restricts severely the range of possible applications. 
Analysing this problem,  in \cite[Corollary 1]{SOTRb}  we have shown, under a quasi-monotonicity condition on $f$,  that the existence of the limit 
\begin{equation}\label{ECM}
\lim_{x \to -\infty} \frac{u_0(s,x)}{\phi(x+cs)}=A_0(s)>0, \quad s \in [-h,0], 
\end{equation}
with some continuous function $A_0(s)$ implies that solution $u(t,x)$ evolves  in the middle of two shifted traveling waves constituting  the lower bound $u_-(t,x)=\phi(x+ct+a_-),$ and the upper bound $u_+(t,x)=\phi(x+ct+a_+)$. Condition (\ref{ECM}) is easily verifiable. Indeed,   it is well known  that under some natural restrictions (tacitly assumed in this work) so-called non-critical waves have the following asymptotic 
representation after an appropriate translation of the time variable:
\begin{eqnarray}\label{t}
\phi(t)&=&e^{\lambda_1 t}+e^{(\lambda_1+\sigma)t} r_1(t), \ \lambda_1+\sigma < \lambda_2, \\
\phi'(t)&=&\lambda_1e^{\lambda_1 t}+e^{(\lambda_1+\sigma)t} r_2(t), \quad\quad t\in {\mathbb R}.\nonumber
\end{eqnarray}
Here $\sigma$ is a positive number, $r_1, \, r_2$ are smooth bounded functions  and  $0<\lambda_1< \lambda_2$ are   zeros 
of the characteristic function $\chi_0(z)= z^2-cz+f_1(0,0)+f_2(0,0)e^{-zch}$. In the paper, $f_j(u,v)$ denotes the partial derivative 
of $f$ with respect to $j$-th argument. We will assume that $f_j(u,v)$ are locally Lipschitz  continuous functions.

A potential possibility that  solution $u(t,x)$ can develop non-decaying oscillations between the waves $u_+(t,x)$ and $u_-(t,x)$ was 
not discarded in  \cite{SOTRb}. Another question left open in  \cite{SOTRb}   is whether such $u(t,x)$ converges to the traveling wave in form and in speed \cite{KPP,UC}, i.e. whether  there exists a function $\beta(t)$ such that $\beta(t)/t \to c$ and $u(t, x-\beta(t)) \to \phi(x)$ as $t \to +\infty$, uniformly on subsets $(-\infty, n],$ $n \in \N$.
In this work, we answer both questions under rather realistic assumptions specified below.

Actually, assuming (\ref{ECM}), we prove that the solution  $u(t,x)$ converges to a shifted wave  $\phi(x+ct+a_*)$, where $a_*$ 
is completely determined by the function $A_0(s)$: 
\begin{equation}\label{as}
a_*=\frac{1}{\lambda_1}\ln A_{\infty}, \quad \mbox{where} \ A_{\infty}:= \frac{A_0(0)+q\int_{-h}^0A_0(s)ds}{1+qh}, \ q := \frac{f_2(0,0)e^{-\lambda_1ch}}{\lambda_1}.
\end{equation}
We obtain $A_{\infty}$ as the limit value at $+\infty$ of the solution $A(t), \ t\geq 0$,  to the initial value problem $A(s) = A_0(s) >0,$  $s \in [-h,0]$, for the monotone scalar delay differential equation 
\begin{eqnarray}\label{al1}
A'(t)&=&q\left(A(t-h)-A(t)\right),\quad\quad t \geq 0. \label{al2}
\end{eqnarray}
Indeed, it is clear that  $A(t) >0$ for all $t \geq -h$.  Since the characteristic equation  $z+q =  qe^{-zh}$ for equation (\ref{al1}) with $f_2(0,0)>0$ has a unique simple real root $z=0$, other (complex) roots $z_j$
satisfying the inequality ${\mathbb R}e \,z_j < 0$ (see Appendix),  there are real numbers $A_\infty \geq 0$ and $d<0$ (cf. \cite[Theorem 3.2]{BC})  such that 
\begin{equation}\label{be}
|A(t) - A_\infty| \leq e^{dt}, \quad t \geq 0. 
\end{equation}
By integrating (\ref{al2}) on ${\mathbb R}_+$, we find that 
$$
A_\infty(1+qh)=A_0(0)+q\int_{-h}^0A_0(s)ds >0. 
$$
Now, (\ref{ECM}),  (\ref{t}) imply that the initial function $u(s,x)$ evaluated at the moment $s=0$ behaves as  $\phi(x+ a_0),$ where $a_0= \ln A_0(0)/\lambda_1$. Therefore  the total traveled distance $\delta_a $ between the initial (at the moment $t=0$) and final (as $t \to +\infty$) positions of the solution in the moving reference frame  is
$$
\delta_a= a_*-a_0=\frac{1}{\lambda_1}\ln  \frac{1+q\int_{-h}^0A_0(s)/A_0(0)ds}{1+qh}. 
$$
Note that the function $A(t)$ and $\delta_a$ are completely determined by the speed $c$, the initial values $A_0(s)$  and the partial derivatives $f_1(0,0), f_2(0,0)$. They do not depend on other characteristics of  solution $u(t,x)$ and wavefront  $\phi(x+ct)$, including their bounds $M_1\leq M_3 \in \R \cup \{+\infty\},$ $M_2\leq 0$, 
$$
0\leq \phi(x)\leq M_1, \quad M_2\leq u(t, x)\leq M_3, \qquad  (t, x)\in[-h, +\infty)\times{\mathbb R}, 
$$
and associated parameters $L_2\geq f_2(0,0) \geq 0$ and $D \in \R$ chosen to satisfy 
$$
|f(w, v_1)-f(w, v_2)|\leq L_2 \, |v_1-v_2|, \quad \quad (w, v_1, v_2)\in [0, M_1]\times[M_2, M_3]^2,
$$
$$
D=\inf_{(w_1,w_2, v)\in[M_2, M_3]^3, w_1\not=w_2} \frac{f(w_1, v)- f(w_2,v)}{w_2-w_1}. 
$$
\begin{remark} \label{R11} Clearly,  $D= 1$ for the Mackey-Glass type nonlinearity $f(w,v)= -w + b(v)$.  Considering monotone wavefronts for the KPP-Fisher delayed equation \cite{BS,ADN,FZ,fhw}, when $f(w,v)= w(1-v)$, we find that 
$L_2= M_1=1$, $M_3=+\infty$. In the general case of non-monotone waves for the latter equation,  we can take $L_2= M_1=e^{ch}$, $M_3=+\infty$, cf. \cite{BS}. 
In both cases (monote and non-monotone), we  have that $D=\inf_{(t,x)\in[0, +\infty)\times{\mathbb R}} u(t, x)-1.$ 
Hence, if $u_0\geq 0=M_2$ then   $D=-1$.
\end{remark}
First, we consider  an easier situation when $f_2(0,0) >0$.
\begin{theorem} \label{Th1} Assume  that $f_2(0,0) >0$ and 
\begin{eqnarray}\label{in1}
\lambda^2-c\lambda -D-\gamma+L_2 e^{-\lambda ch} e^{-\gamma h} <0,
\end{eqnarray}
for some  $\lambda\in (\lambda_1, \min\{2\lambda_1,\lambda_2\})$  and $\gamma \in (d,0)$.  If, in addition,  
$u_0(s,x)$ verifies
\begin{eqnarray}\label{IC1}
|u_0(s,x)-\phi(x+cs+\alpha_0(s))|\leq  Ke^{\lambda x},\qquad  (s,x)\in[-h,0]\times{\mathbb R},
\end{eqnarray}
then, for some $K'\geq K$, solution $u(t,x)$ of  (\ref{E}) with the initial function $u_0$ satisfies 
\begin{eqnarray}\label{conv}
\sup_{x \in \R}\left(e^{-\lambda x}|u(t,x-ct-\alpha(t))-\phi(x)|\right)\leq K' e^{\gamma t},\quad  t\geq -h.
\end{eqnarray}
Here $A(t)=e^{\lambda_1\alpha(t)}$ solves 
(\ref{al1}) with the initial datum $A_0(s)=e^{\lambda_1\alpha_0(s)},$ $s \in [-h,0]$ so that 
$
\alpha(+\infty)=a_* \in [\min_{[-h,0]}\alpha_0(s), \max_{[-h,0]}\alpha_0(s)]
$
is given by (\ref{as}). Finally,  $\delta_a=0$ if and only if $A(0) = (1/h)\int_{-h}^0A_0(s)ds$. 
\end{theorem}
Next, we consider the `degenerate' situation when $f_2(0,0)=0$. From (\ref{al1}), we can expect that $\alpha(t) \equiv \alpha(0)$ 
for $t \geq 0$. Below, we prove that this is indeed the case for a class of the KPP-Fisher type nonlinearities. 

\begin{theorem} \label{Th1a} Assume  that $f(u,v)=g(u)(\kappa -v)$ with $\kappa >0$, $g(0)=0$, that (\ref{in1}) holds
for some  $\lambda\in (\lambda_1, \lambda_2)$   and $\gamma <0$, and that 
 $u_0(s,x)$  satisfies,   for some $\lambda^*> \lambda_1$,  
\begin{eqnarray}\label{IC1a}
|u_0(s,x)-\phi(x+cs+\alpha_0(s))|\leq  Ke^{ \lambda^* x},\qquad  (s,x)\in[-h,0]\times{\mathbb R}.
\end{eqnarray}
Set $\lambda_* =\min\{\lambda^*, \lambda, 2\lambda_1\}$. If 
$\lambda_*^2-c\lambda_* -D-\gamma <0$, 
then, for some $K'\geq K$, solution $u(t,x)$ of  equation (\ref{E}) with the initial function $u_0(s,x)$ satisfies 
\begin{eqnarray}\label{convb}
\sup_{x \in \R}\left(e^{-\lambda_* x}|u(t,x-ct-\alpha_0(0))-\phi(x)|\right)\leq K' e^{\gamma t},\quad  t\geq -h.
\end{eqnarray}
\end{theorem}

\begin{figure}[h] \label{F41}
\centering {\includegraphics[width=6.3cm]{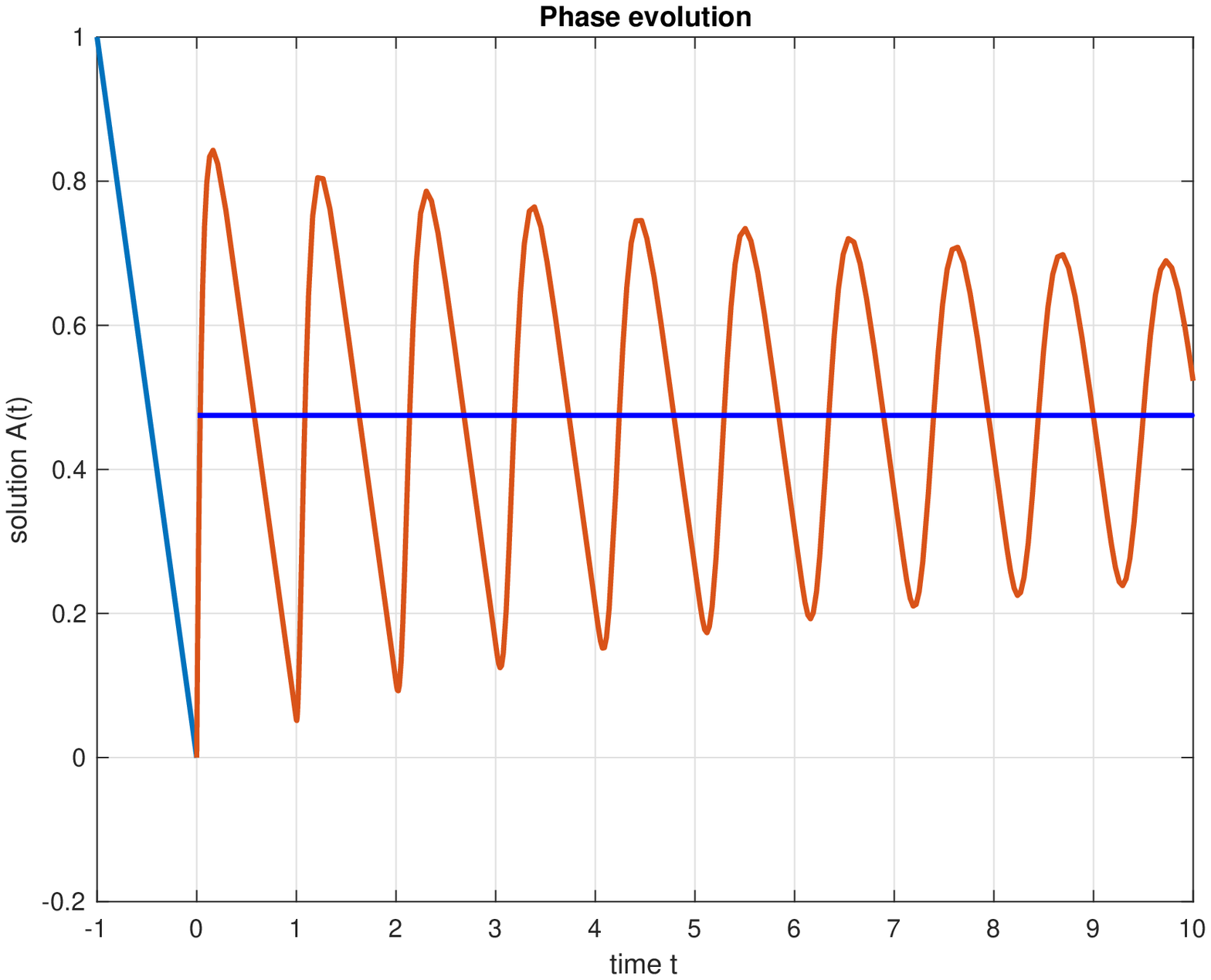}\includegraphics[width=7.3cm]{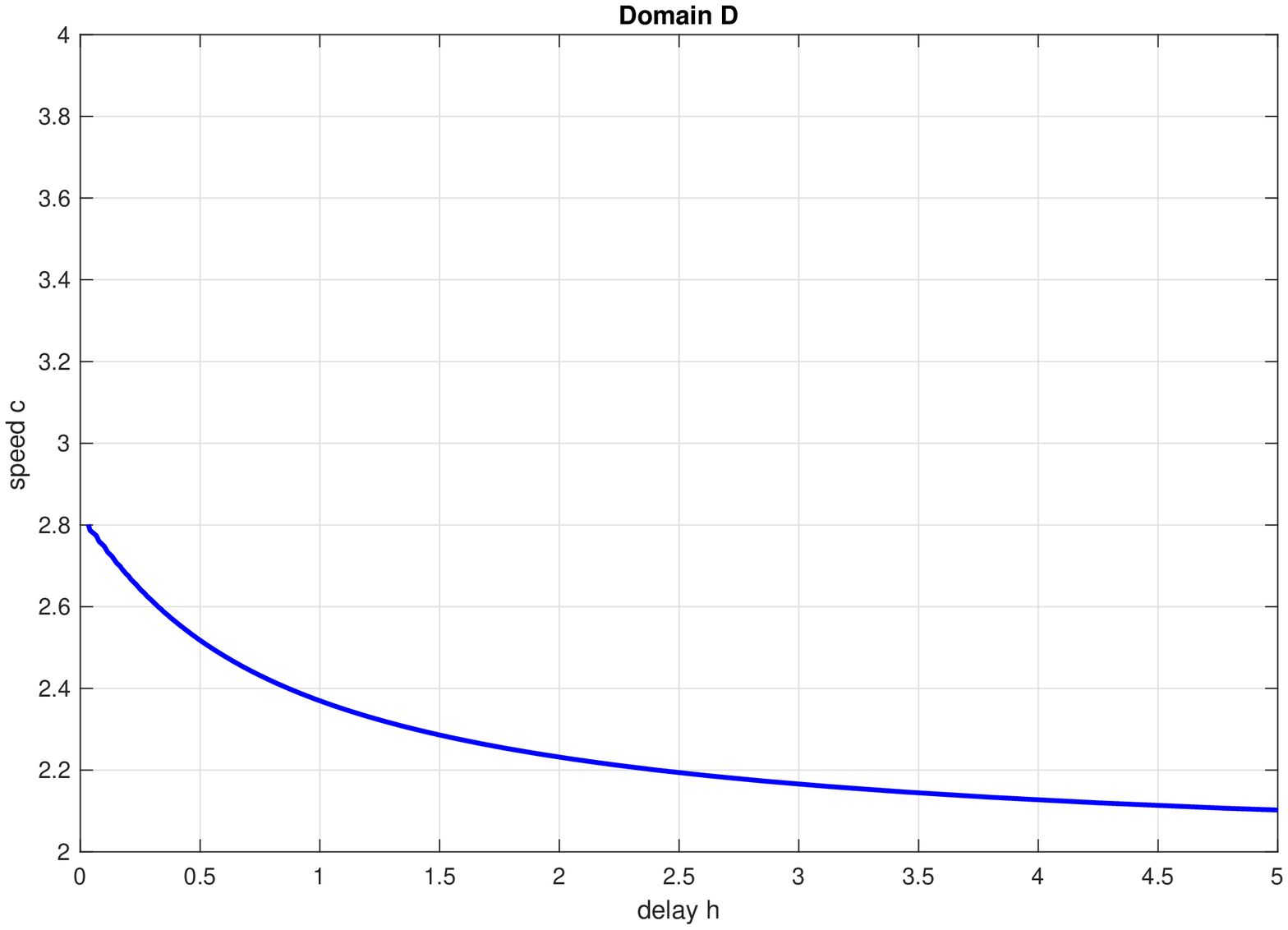}}
\caption{\hspace{0cm}  On the left: particular solution
of (\ref{al1}) with $h=1$,  $q=19$, $A_0(s)=-s$. Horizontal line is the limit value $A_\infty= 19/40$. On the right, the graph of  $c= c_\#(h)$  from  Corollary \ref{Th2C}.} 
\end{figure} 
\vspace{0mm}

Theorems \ref{Th1} and  \ref{Th1a}  say that the evolution of the initial phase deviation $\alpha_0(s)$ is determined by the  linear delay differential equation (\ref{al1}).  More detailed analysis of  the eigenvalues $z_j$ to (\ref{al1}) (see the Appendix) allows to have a better idea about the character of convergence of  $\alpha(t)$ to its limit $\alpha(+\infty)$. We claim that,  in the non-degenerate case  $hf_2(0,0)\not =0$, $\alpha_0(s)\not\equiv const$, generically  $\alpha(t)$ develops `rapid'  oscillations around $\alpha(+\infty)$ (these oscillations can be significant when $q$ is relatively large, see Figure 1). More precisely,  generically 
$\alpha(t)$ crosses two times the level  $\alpha(+\infty)$ on each half-open interval of the length $h$.  Indeed, an application of the Laplace transform to (\ref{al1}) yields the following representation 
$$
A(t) = A_\infty+ 2\Re (A_1e^{z_1t})(1 + o(1)), \ \mbox{where} $$
$$
A_1(1+h(z_1+q))=A_0(0)+qe^{-z_1h}\int_{-h}^0e^{-z_1s}A_0(s)ds 
$$
with $z_1=x_1+i y_1, y_1h \in (\pi, 2\pi)$, being the leading complex eigenvalue of (\ref{al1}). 

In particular,  $\alpha(t)$ is typically oscillating in the case of Nicholson's diffusive equation \cite{Chern,LvW,SOTRb,wlr}
\begin{equation}\label{MGE} 
u_t (t,x)=u_{xx}(t,x)- u(t,x) + b(u(t-h, x)), \  \  x\in{\mathbb R}, \quad b(u)=pue^{-u},\ p>1, 
\end{equation}
In such a case, $L_2= b'(0)=p$, $D=1$, and the solution $u(t,x),$  $t\geq 0, x \in \R,$ is bounded once its initial fragment  $u_0(s,x)$,  $s\in [-h,0], x \in \R,$ is bounded. In addition, the formulae (\ref{t}) hold for each $c > c_*$, where $c_*$ is the minimal speed of propagation in the model. In this way, we obtain  the following conclusion:  
\begin{cor} \label{Cor} Let $u=\phi(x+ct)$ be a non-critical wave for the Nicholson's diffusive equation. Denote by  $u(t,x)$  solution of the initial   problem $u(s,x)= u_0(s,x),$ $s\in [-h,0],$ for (\ref{MGE}) where
non-negative  function  $u_0$ satisfies 
\begin{eqnarray}\label{IC1M}
|u_0(s,x)-\phi(x+cs +\alpha_0(s))|\leq  Ke^{\lambda x},\qquad  (s,x)\in[-h,0]\times{\mathbb R},
\end{eqnarray}
for some $\lambda > \lambda_1$. 
Then there exist $\mu \in (\lambda_1,\lambda]$, $\gamma <0$ and  $Q\geq K$ such that 
\begin{eqnarray*}
|u(t,x)-\phi(x+ct+\alpha(t))|\leq Q e^{\mu (x+ct)} e^{\gamma t},\quad  x\in{\mathbb R}, \ t\geq -h.
\end{eqnarray*}
The function $\alpha(t)$ is converging at $+\infty$ and generically  develops `rapid'  oscillations around its limiting value  $\alpha(+\infty)$.
 \end{cor}
Other aforementioned model, the KPP-Fisher delayed equation 
\begin{equation}\label{KPP} 
u_t (t,x)=u_{xx}(t,x)+ u(t,x)(1-u(t-h, x)), \ u= u_0(s,x), s\in [-h,0], \  x\in{\mathbb R}, 
\end{equation}
has the reaction term satisfying the equality $f_2(0,0)=0$.  
In view of Theorem  \ref{Th1a} and  Remark \ref{R11}, in the general case of non-monotone waves we have to consider the  domain ${\mathcal D}$ (presented on the right panel of Fig. 1 as a strict epigraph for the decreasing function $c=c_\#(h), \ $ $h \geq 0,$ $\ c_\#(0)=2\sqrt{2}, \ c_\#(+\infty)=2$), 
$$
{\mathcal D} = \left\{(h,c): \lambda^2-c\lambda +1+e^{-\lambda ch+ch} <0 \ \mbox{for some}\ \lambda \right\} = \left\{(h,c): c> c_\#(h), \ h \geq 0\right\},
$$
where $c=c_\#(h), \ h \geq 0,$ is defined implicitly by 
$$
 -2+\sqrt{c^4h^2-4c^2h^2+4}-c^2h^2\exp (ch\left(1-\frac c 2+\frac{1}{ch}-\sqrt{\frac{c^2}{4}+\frac{1}{c^2h^2}-1}\right))=0.
$$

Then Theorem \ref{Th1a}   yields  the following conclusion.
\begin{cor} \label{Th2C} Let $u=\phi(x+ct)$ be a traveling wave for KPP-Fisher delayed equation  (\ref{KPP}) where $(h, c) \in {\mathcal D}$. 
 Denote by  $u(t,x)$  solution of the initial   problem (\ref{KPP}) where
non-negative  function  $u_0$, satisfies,  for some $\lambda^* > \lambda_1$,
\begin{eqnarray}\label{IC1M}
|u_0(s,x)-\phi(x+cs +\alpha_0(s))|\leq  Ke^{\lambda^* x},\qquad  (s,x)\in[-h,0]\times{\mathbb R}.
\end{eqnarray}
Then there exist $\mu \in (\lambda_1,2\lambda_1)$, $\gamma <0$ and  $Q\geq K$ such that 
\begin{eqnarray}\label{conva}
|u(t,x)-\phi(x+ct+\alpha_0(0))|\leq Q e^{\mu (x+ct)} e^{\gamma t},\quad  x\in{\mathbb R}, \ t\geq 0.
\end{eqnarray}
 \end{cor}
In this way, on the base of an alternative  approach,   Theorem  \ref{Th1a}  and Corollary \ref{Th2C} improve the stability result \cite[Theorem 3]{BS} in the following two aspects: a) in Corollary \ref{Th2C}, the 
initial phase function $\alpha_0(s), s \in [-h,0]$ is not necessarily constant; b) even if all mentioned results use the same domain for the admissible parameters $(h, c)$,   \cite[Theorem 3]{BS}  assumes additionally that the exponent $\lambda^*$ in (\ref{IC1M}) should be larger than some minimal value, specific  for each pair $(h,c)$.  Observe that for the delayed KPP-Fisher equation it is still not clear whether  a) the domain of all admissible parameters can be extended to the quarter-plane $c \geq 2, h \geq 0$; b)  the estimate (\ref{conva}) with the bounded weight $\min\{e^{\mu x},1\}$ is true.   
\section{Proof of Theorem \ref{Th1}} \label{Section 2}
The estimation of the auxiliary function 
$$P=f(\phi(x+\alpha(t)), \phi(x-ch+\alpha(t)))-f(\phi(x+\alpha(t)), \phi(x-ch+\alpha(t-h)))+\alpha'(t) \phi'(x+\alpha(t))$$
 is instrumental  for proving our first main result. 
\begin{lemma}\label{L2} Assume all  conditions of Theorem \ref{Th1}. Let  $q$ and $d$ be defined by (\ref{as}) and  (\ref{be}),
respectively.  Then  $|P(t,x)| \leq q_0\, e^{\lambda x}\, e^{d t}$  for some $q_0\geq 0$ and all $x\in{\mathbb R}$, $t\geq 0$.   
\end{lemma}
\begin{proof} We have that 
$$
P(t, x)=\left(\alpha'(t) \phi'(x+\alpha(t))+
f_2(0,0)[\phi(x-ch+\alpha(t))-\phi(x-ch+\alpha(t-h))]\right)+$$
$$\rho[\phi(x-ch+\alpha(t))-\phi(x-ch+\alpha(t-h))] =:{\mathcal P}_1+{\mathcal P}_2,
$$
where
$$
\rho= f_2(\phi(x+\alpha(t)), \theta(x,t))-f_2(0,0)
$$ 
with $\theta(x,t)$ being some point  between $\phi(x-ch+\alpha(t))$ and $\phi(x-ch+\alpha(t-h))$. 
Since $|f_2(u,v)-f_2(0,0)| \leq C(|u|+|v|)$ on the bounded subset  $[0,M_1]^2\subset \R_+^2$
 (in this proof, we are using $C$ as a generic positive constant), we conclude that 
 $$|{\mathcal P_2}|= |\rho\left[\phi(x-ch+\alpha(t))-\phi(x-ch+\alpha(t-h))\right]| \leq $$
$$Ce^{\lambda_1 x}|\phi(x-ch+\alpha(t))-\phi(x-ch+\alpha(t-h))| \leq 
Ce^{\lambda_1 x}e^{\lambda_1x +dt} \leq Ce^{\lambda x +dt}, \ x \in {\mathbb R}, \ t \geq 0.
$$
Next, consider   $B(z)= \phi(z)-e^{\lambda_1 z}$, clearly $B(z)= O(e^{\lambda z})$ at $z=-\infty$. Then (\ref{t}) and (\ref{al1}) imply that 
$$
{\mathcal P}_2= f_2(0,0)[\phi(x-ch+\alpha(t))-\phi(x-ch+\alpha(t-h))]+{\alpha}'(t) \phi'(x+\alpha(t))= $$
$$ e^{\lambda_1 x} [f_2(0,0) (e^{\lambda_1(\alpha(t)-ch)}-e^{\lambda_1(\alpha(t-h)-ch)})+\lambda_1{\alpha}'(t) e^{\lambda_1 \alpha(t)}]
+ $$
$$f_2(0,0)[B(x-ch+\alpha(t))-B(x-ch+\alpha(t-h))]+{\alpha}'(t)r_2(x+\alpha(t)) e^{(\lambda_1+\sigma)(x+\alpha(t))}
=$$
$$f_2(0,0)[B(x-ch+\alpha(t))-B(x-ch+\alpha(t-h))]+{\alpha}'(t)r_2(x+\alpha(t)) e^{(\lambda_1+\sigma)(x+\alpha(t))}.
$$
\vspace{2mm}
Next, we have that 
$$|\alpha(t)-\alpha(t-h)| = \lambda_1^{-1}\left|\ln\frac{A(t)}{A(t-h)}\right|
\leq C |A(t)-A(t-h)| \leq  Ce^{d t}, t \geq 0.$$
As a consequence, 
$$|B(x-ch+\alpha(t))-B(x-ch+\alpha(t-h))|\leq Ce^{\lambda x+d t},\qquad x\in {\mathbb R}, \ t\geq 0.$$   
In this way, since  $\alpha'(t)= \lambda_1^{-1}A'(t)/A(t) = O(e^{dt}),\ t \to +\infty,$ we find that 
$|{\mathcal P}_2| \leq Ce^{\lambda x+d t}, \ x\in {\mathbb R}, \, t\geq 0.
$
The obtained estimates for $|{\mathcal P}_1|$ and $|{\mathcal P}_2|$ show that, for some positive constant $q_0$,  $$|P(t,x)| \leq q_0\, e^{\lambda x+d t}, \ x \in {\mathbb R}, \ t >0.$$
This completes the proof of Lemma \ref{L2}.  
\end{proof}
\begin{proof}[Proof of Theorem \ref{Th1}]  Set
$v(t,x)= u(t,x-ct)$. Then the problem (\ref{E}), (\ref{IC1}) takes the form
$$
0=v_{xx}(t,x)- cv_x(t,x) - v_t (t,x)+f(v(t, x),v(t-h, x-ch)), \quad t>0, \  x\in{\mathbb R},
$$
$$
|v_0(s,x)-\phi(x+\alpha_0(s))|\leq  Ke^{ \lambda x},\qquad  (s,x)\in[-h,0]\times{\mathbb R}.
$$

Take $q_0$ as in Lemma \ref{L2} and let $Q\geq K$ be sufficiently large to satisfy 
$$
-\lambda^2+c\lambda+D-L_2\, e^{-\lambda ch}e^{-\gamma h}- \frac{q_0}{Q}+\gamma> 0. 
$$
\noindent For $\phi(x+\alpha(t))\not= v(t, x)$,  set  $$d(t,x):=\frac{f(\phi(x+\alpha(t)), v(t-h, x-ch))-f(v(t, x), v(t-h, x-ch))}{\phi(x+\alpha(t))- v(t, x)}, $$
and for  $\phi(x+\alpha(t))= v(t, x)$,  set $d(t,x):=f_1(\phi(x+\alpha(t)), v(t-h, x-ch))$. 

Then consider the linear differential operator $$\mathcal{L}v=v_{xx}-cv_x+d(t,x)v-v_t $$
and the functions 
 $$\delta_{\pm}(t,x)=\pm[v(t, x)-\phi(x+\alpha(t))]-Qe^{\gamma t}e^{\lambda x}.$$  
 By our assumptions $\delta_{\pm}(t,x)\leq 0$ for   $(t,x)\in[-h,0]\times{\mathbb R}$. Let $\Pi= [-h,T]\times{\mathbb R}$, $T \in \R_+\cup \{+\infty\}$,
 be the maximal strip where $\delta_{\pm}(t,x)\leq 0$.  Clearly, inequality (\ref{conv}) is satisfied for all $(t,x) \in \Pi$. Theorem \ref{Th1} will be proved if we establish that $T =+\infty$. 
 Suppose for a moment that $T$ is finite. Then 
 we find that, for all $t \in [T,T+h]$, $x \in \R$, 
$$
(\mathcal{L}\,\delta_{\pm})(t, x)=\{\pm(\mathcal{L}\,v)(t, x)\mp(\mathcal{L}\,\phi(\cdot +\alpha))(t, x)\}-Q\, (\mathcal{L} e^{\gamma \cdot}e^{\lambda \cdot})(t,x)=$$
$$
\pm\Big\{f(\phi(x+\alpha(t)), \phi(x-ch+\alpha(t)))-f(v(t, x), v(t-h, x-ch))+\alpha'(t) \phi'(x+\alpha(t))$$
$$-d(t, x)[\phi(x+\alpha(t))-v(t,x))]\Big\} + Q e^{\lambda x}e^{\gamma t}[-\lambda^2+c\lambda-d(t, x)+\gamma]=$$
$$\pm[f(\phi(x+\alpha(t)), \phi(x-ch+\alpha(t)))-f(\phi(x+\alpha(t)), \phi(x-ch+\alpha(t-h)))+\alpha'(t) \phi'(x+\alpha(t))]$$
 $$\pm\{f(\phi(x+\alpha(t)), \phi(x-ch+\alpha(t-h)))-f(v(t, x),  v(t-h, x-ch))- d(t, x)[\phi(x+\alpha(t))-v(t,x)]\}$$
$$ +Qe^{\lambda x}e^{\gamma t}[-\lambda^2+c\lambda-d(t, x)+\gamma]\geq $$
$$ -q_0\, e^{\lambda x} e^{\gamma t}\pm[f(\phi(x+\alpha(t)), \phi(x-ch+\alpha(t-h)))-f(\phi(x+\alpha(t)), v(t-h, x-ch))]+$$ 
$$Qe^{\lambda x}e^{\gamma t}[-\lambda^2+c\lambda-d(t, x)+\gamma]\geq $$
$$ -q_0\, e^{\lambda x} e^{\gamma t}-L_2 \, |\phi(x-ch+\alpha(t-h))-v(t-h, x-ch)|+Qe^{\lambda x}e^{\gamma t}[-\lambda^2+c\lambda-d(t, x)+\gamma]=$$
$$Q e^{\lambda x}e^{\gamma t}[-\frac{q_0}{Q}-L_2e^{-\lambda ch}e^{-\gamma h}-\lambda^2+c\lambda+D+\gamma]\geq 0. 
$$
Invoking the Phragm\`en-Lindel\"of principle at this stage, we conclude that  also  $\delta_{\pm}(t, x)\leq 0$ for all $t\in[T,T+h]$, $x \in \R$. This contradicts the maximality   of the strip $\Pi$ and completes the proof of the theorem. 
\end{proof}

\section{Proof of Theorem \ref{Th1a}}  The change of variables 
$v(t,x)= u(t,x-ct)$ transforms  (\ref{E}), (\ref{IC1a}) into 
$$
0=v_{xx}(t,x)- cv_x(t,x) - v_t (t,x)+f(v(t, x),v(t-h, x-ch)), \quad t>0, \  x\in{\mathbb R},
$$
$$
|v_0(s,x)-\phi(x+\alpha_0(s))|\leq  Ke^{ \lambda^* x},\qquad  (s,x)\in[-h,0]\times{\mathbb R}.
$$
Without loss of generality, we  can assume that  $\alpha_0(0)=0$. 
Our first goal is to obtain a similar estimate for  $t\in [0,h]$: we will prove that, for some $K_1 \geq K$,  
\begin{eqnarray}\label{IC1aPP}
|v(t,x)-\phi(x)|\leq  K_1e^{ \lambda_* x},\qquad  (t,x)\in[0,h]\times{\mathbb R}.
\end{eqnarray}

Indeed,  the difference $w(t,x)= v(t,x)-\phi(x)$ solves 
the following linear inhomogeneous equation 
\begin{eqnarray*}
w_t(t, x)&=& w_{xx}(t,x)- cw_x(x,t) + a(t,x)w(t, x) +b(t,x),  \  \ t\in[0,h], \, x\in\R, \\
w(s, x)&=& w_0(s,x):=v_0(s,x)-\phi(x), \qquad (s, x)\in[-h, 0]\times\R,
\end{eqnarray*}
where 
$$
a(t,x)=\int_0^1f_1(sv(t,x)+(1-s)\phi(x), sv(t-h,x-ch)+(1-s)\phi(x-ch))ds
$$
$$b(t,x)= -w(t-h,x-ch)\int_0^1g(sv(t,x)+(1-s)\phi(x))ds$$
are Lipschitz  continuous functions. 
Invoking the standard representation formula for the solution of the above Cauchy problem (see \cite[Theorem 12]{AF}), we find that, for
$(t, x)\in[0, h]\times\R$ it holds 
\begin{eqnarray*}
w(t, x)=  \int_{\R}\Gamma(t, x; 0, \xi )w(0, \xi) d\xi + \int_0 ^t \int_{\R}\Gamma(t, x; \tau, \xi )b(\tau,\xi) d\xi d\tau, 
\end{eqnarray*}
where  $\Gamma(t, x; \tau, \xi )$ is the fundamental solution for the respective homogeneous equation. Using the estimates (for the first one, see  inequality (6.12) on p. 24 of \cite{AF})
$$
|\Gamma(t, x; \tau, \xi )| \leq\frac{C}{\sqrt{t-\tau}} e^{-\frac{k(x-\xi)^2}{4(t-\tau)}},\quad  x, \xi \in {\mathbb R}, \ t > \tau, \ t, \tau \in [0,h], 
$$
$$
|b(\tau,\xi)| +
|w(0, \xi)| \leq  Ce^{\lambda_* \xi}, \qquad  (\tau,\xi)\in[0,h]\times{\mathbb R},
$$
where $C>0$ and $k\in (0,1)$ are some constants, we obtain, with some $C'>0$,  that 
$$
e^{-\lambda_* x} |\int_{\R} \Gamma(t, x; 0, \xi)\, w(0, \xi) d\xi|\leq  \int_{\R} \frac{C^2}{\sqrt{t}} e^{-\frac{k(x-\xi)^2}{4t}} e^{-\lambda_*(x-\xi)}d\xi= $$
$$
=2C^2 \int_{\R} e^{-ks^2} \, e^{2\lambda_*s \sqrt{t}} ds \leq C', \quad t \in [0,h], \ x \in \R,$$
$$
e^{-\lambda_*x} \left| \int_0 ^t  \int_{\R}\Gamma(t, x; \tau, \xi ) b(\tau, \xi) d\xi  d\tau \right|\leq 
e^{-\lambda_* x} \left| \int_0 ^t  \int_{\R}\frac{C^2}{\sqrt{t-\tau}} e^{-\frac{k(x-\xi)^2}{4(t-\tau)}}e^{\lambda_*\xi}d\xi  d\tau \right|=
$$
$$
| \int_0 ^t  \int_{\R}\frac{C^2}{\sqrt{t-\tau}} e^{-\frac{k\xi^2}{4(t-\tau)}}e^{\lambda_*\xi}d\xi  d\tau| < C', \quad t \in [0,h], \ x \in \R. 
$$
Then (\ref{IC1aPP}) follows from these inequalities.  

Next, take a sufficiently large  negative number $x_*$ to have 
$$
|g(\phi(x))| < -0.25\gamma e^{0.5\gamma h} \quad \mbox{for all} \ x \leq x_*. 
$$
Consider a $C^\infty$-smooth non-decreasing function $\lambda: \R \to \R$ defined, for some appropriate $\theta >ch,$ as
$\lambda(x)= \lambda_*x$ for $x \leq x_*-\theta$ and  $\lambda(x)= \lambda x$ for $x \geq x_*-ch$ and $\lambda'(x) \in [\lambda_*,\lambda]$, $\lambda''(x) < -\gamma/4$. 
Clearly, we can choose $K_2>K_1$ in such a way that the functions 
 $$\rho_{\pm}(t,x)=\pm[v(t, x)-\phi(x)]-K_2e^{0.5\gamma (t-h)}e^{\lambda(x)}$$  
satisfy $\rho_{\pm}(t,x)\leq 0$ for   $(t,x)\in[0,h]\times{\mathbb R}$. 
 
\noindent For $\phi(x)\not= v(t, x)$,  set  $$m(t,x):=\frac{f(\phi(x), v(t-h, x-ch))-f(v(t, x), v(t-h, x-ch))}{\phi(x)- v(t, x)}, $$
and for  $\phi(x)= v(t, x)$,  set $m(t,x):=f_1(\phi(x), v(t-h, x-ch))$. 

Then consider the linear differential operator $$\mathcal{L}v=v_{xx}-cv_x+m(t,x)v-v_t $$
and let $\Pi= [0,T]\times{\mathbb R}$, $T \in [h, +\infty]$
 be the maximal strip where $\rho_{\pm}(t,x)\leq 0$. 
 Suppose for a moment that $T$ is finite. Then 
 we find that, for all $t\in[T,T+h]$, $x \in \R$, 
$$
(\mathcal{L}\,\rho_{\pm})(t, x)=\{\pm(\mathcal{L}\,v)(t, x)\mp(\mathcal{L}\,\phi(\cdot))(t, x)\}-K_2e^{-0.5\gamma h}\, (\mathcal{L} e^{0.5\gamma \cdot}e^{\lambda(\cdot)})(t,x)=$$
 $$\pm\{f(\phi(x), \phi(x-ch))-f(v(t, x),  v(t-h, x-ch))- m(t, x)[\phi(x)-v(t,x)]\}$$
$$ + K_2 e^{\lambda (x)}e^{0.5\gamma (t-h)}[-\lambda''(x) - (\lambda'(x))^2+c\lambda'(x)-m(t, x)+0.5\gamma]\geq $$
$$ -|g(\phi(x))| |\phi(x-ch)-v(t-h, x-ch)|+$$
$$K_2e^{\lambda (x)}e^{0.5\gamma (t-h)}[-\lambda''(x) - (\lambda'(x))^2+c\lambda'(x)+D+0.5\gamma]= :{\mathcal E}(t,x). $$
Now, if $x \leq x_*$ then
$$
{\mathcal E}(t,x)\geq -|g(\phi(x))|K_2e^{0.5\gamma (t-2h)}e^{\lambda(x-ch)}+$$
$$K_2 e^{\lambda (x)}e^{0.5\gamma (t-h)}[-\lambda''(x) - (\lambda'(x))^2+c\lambda'(x)+D+0.5\gamma]\geq 
$$
$$K_2e^{\lambda (x)} e^{0.5\gamma (t-h)}[\gamma - (\lambda'(x))^2+c\lambda'(x)+D]>0. 
$$
On the other hand, if $x \geq x_*$ then $${\mathcal E}(t,x)\geq 
K_2e^{\lambda(x)} e^{0.5\gamma (t-h)}[ -L_2e^{-0.5\gamma h}e^{-\lambda ch}-\lambda^2+c\lambda+D+\gamma]>0.
$$

Invoking the Phragm\`en-Lindel\"of principle at this stage, we conclude that  also  $\delta_{\pm}(t, x)\leq 0$ for all $t\in[T,T+h]$, $x \in \R$. This contradicts the maximality   of the strip $\Pi$ and completes the proof of the theorem. \qed
\section*{Appendix} 
Here we analyse the zeros of the entire function $z+q - qe^{-zh}$, where  $q, h$ are positive parameters. It is convenient to include the case 
$q=+\infty$ by introducing $\epsilon =1/q\geq 0$ and analysing $\chi(z)=\epsilon z+1 - e^{-zh}$. 
Clearly, $\chi$ has only one  real 
zero $z=0$. Thus $\chi'(z_j)= \epsilon+h(\epsilon z_j+1)\not=0$ at each zero $z_j$ of $\chi(z)$ so that  $z_j=z_j(\epsilon)$ is a smooth function of $\epsilon \geq 0$.  
Set $z_j=x+iy$ with $y>0$, then  
$\epsilon x +1 = e^{-xh}\cos(yh)$, $\epsilon y= -e^{-xh}\sin(yh)$ and therefore 
the unique zero of $\chi(z)$ with non-negative real part is $z=0$. Moreover, the equality $\epsilon y=-e^{-xh}\sin(yh)$  shows that $yh \in (\pi+ 2\pi k, 2\pi+ 2\pi k)$,  $k\in \N \cup \{0\},$ whenever $\epsilon >0$.  Next,  $1 =e^{-z_j(0)h}$ implies that $z_j(0)h = i(\pi  + 2\pi k)$. Since  the relation $z_j(\epsilon_*-) = \infty$ cannot happen for a finite $\epsilon_*>0$, we conclude that $z_j(\epsilon)\in  \{z: h\Im z \in (\pi+ 2\pi k, 2\pi+ 2\pi k), \Re z <0\}$ is well defined for every $\epsilon >0$. Consequently, the original function $z+q - qe^{-zh}$ has a unique zero $z_k$ at each horizontal strip $(\pi+ 2\pi k)/h < \Im z<  (2\pi+ 2\pi k)/h$ while its complete list of zeros is given by 
$\{z_0=0, z_k, \bar z_k, k \in \N\}$.   Since $|z_j+q| = qe^{-\Re z_j h}$ we conclude that $\Re z_j$ is a strictly decreasing sequence converging to $-\infty$. 

\bibliographystyle{amsplain}

\end{document}